\documentclass[reqno, 12pt]{amsart}
\pdfoutput=1
\makeatletter
\let\origsection=\section \def\section{\@ifstar{\origsection*}{\mysection}}
\def\mysection{\@startsection{section}{1}\z@{.7\linespacing\@plus\linespacing}{.5\linespacing}{\normalfont\scshape\centering\S}}
\makeatother

\usepackage{amsmath,amssymb,amsthm}
\usepackage{mathrsfs}
\usepackage{mathabx}\changenotsign
\usepackage{dsfont}

\usepackage{xcolor}

\usepackage{tikz}
\usetikzlibrary{calc, positioning, shapes}

\usepackage[backref]{hyperref}
\hypersetup{
    colorlinks,
    linkcolor={red!60!black},
    citecolor={green!60!black},
    urlcolor={blue!60!black}
}

\usepackage[open,openlevel=2,atend]{bookmark}

\usepackage[abbrev,msc-links,backrefs]{amsrefs}
\usepackage{doi}

\renewcommand{\PrintDOI}[1]{\doi{#1}}

\usepackage[T1]{fontenc}
\usepackage{lmodern}
\usepackage[babel]{microtype}
\usepackage[english]{babel}

\usepackage{geometry}
\numberwithin{equation}{section}
\numberwithin{figure}{section}

\usepackage{enumitem}
\def\rmlabel{\upshape({\itshape \roman*\,})}

\def\alabel{\upshape({\itshape \alph*\,})}

\let\polishlcross=\l
\def\l{\ifmmode\ell\else\polishlcross\fi}

\def\tand{\ \text{and}\ }

\def\qqand{\qquad\text{and}\qquad}

\let\emptyset=\varnothing
\let\setminus=\smallsetminus
\let\backslash=\smallsetminus
\let\sm=\setminus

\makeatletter
\def\moverlay{\mathpalette\mov@rlay}
\def\mov@rlay#1#2{\leavevmode\vtop{   \baselineskip\z@skip \lineskiplimit-\maxdimen
   \ialign{\hfil$\m@th#1##$\hfil\cr#2\crcr}}}
\newcommand{\charfusion}[3][\mathord]{
    #1{\ifx#1\mathop\vphantom{#2}\fi
        \mathpalette\mov@rlay{#2\cr#3}
      }
    \ifx#1\mathop\expandafter\displaylimits\fi}
\makeatother

\newcommand{\dcup}{\charfusion[\mathbin]{\cup}{\cdot}}

\DeclareFontFamily{U}  {MnSymbolC}{}
\DeclareSymbolFont{MnSyC}         {U}  {MnSymbolC}{m}{n}
\DeclareFontShape{U}{MnSymbolC}{m}{n}{
    <-6>  MnSymbolC5
   <6-7>  MnSymbolC6
   <7-8>  MnSymbolC7
   <8-9>  MnSymbolC8
   <9-10> MnSymbolC9
  <10-12> MnSymbolC10
  <12->   MnSymbolC12}{}
\DeclareMathSymbol{\powerset}{\mathord}{MnSyC}{180}

\let\epsilon=\varepsilon
\let\eps=\epsilon
\let\rho=\varrho
\let\theta=\vartheta

\let\phi=\varphi

\def\NN{{\mathds N}}

\def\ZZ{{\mathds Z}}
\def\PP{{\mathds P}}
\def\RR{{\mathds R}}

\newcommand{\cF}{\mathcal{F}}

\newcommand{\cQ}{\mathcal{Q}}
\newcommand{\cR}{\mathcal{R}}

\newcommand{\ccA}{\mathscr{A}}

\newcommand{\ccC}{\mathscr{C}}
\newcommand{\ccF}{\mathscr{F}}
\newcommand{\ccG}{\mathscr{G}}

\newcommand{\ccT}{\mathscr{T}}

\theoremstyle{plain}
\newtheorem{thm}{Theorem}[section]

\newtheorem{prop}[thm]{Proposition}
\newtheorem{clm}[thm]{Claim}

\newtheorem{lemma}[thm]{Lemma}

\theoremstyle{definition}

\newtheorem{dfn}[thm]{Definition}

\newtheorem{question}[thm]{Question}

\usepackage{accents}

\def\dhom{\delta_{{\rm hom}}}
\def\dhomp{\delta'_{{\rm hom}}}
\def\dchi{\delta_{\chi}}
\def\ahom{\xrightarrow{{\,\rm hom\,}}}

\begin{document}

\title[Homomorphism thresholds for odd cycles]{Homomorphism thresholds for odd cycles}

\author{Oliver Ebsen}
\author{Mathias Schacht}
\address{Fachbereich Mathematik, Universit\"at Hamburg, Hamburg, Germany}
\email{Oliver.Ebsen@uni-hamburg.de}
\email{schacht@math.uni-hamburg.de}
\thanks{The second author is supported by ERC Consolidator Grant 724903.}

\subjclass[2010]{05C35 (05C07, 05C15, 05D40)}
\keywords{extremal graph theory, odd cycles, homomorphism threshold}

\begin{abstract}
The interplay of minimum degree conditions and structural properties of large
graphs with forbidden subgraphs is a central topic in extremal
graph theory. For a given graph $F$ we define the homomorphism threshold
as the infimum over all~$\alpha\in[0,1]$ such that every $n$-vertex $F$-free graph $G$
with minimum degree at least $\alpha n$ has a homomorphic image $H$ of bounded
order
(i.e.\ independent of $n$), which is $F$-free as well. Without the
restriction of $H$ being $F$-free we recover the definition of the
chromatic threshold, which was determined for every graph $F$ by Allen
et al.\ [\textsl{Adv.\ Math.}~235 (2013), 261--295]. The homomorphism threshold is less understood and we 
address the problem for odd cycles.
\end{abstract}

\maketitle

\section{Introduction}
\label{sec:intro}
Many questions in extremal graph theory concern the interplay of minimum degree conditions and 
structural properties of large graphs with forbidden subgraphs (see, e.g.,~\cites{Z47,A64,AES74}). For a family of graphs~$\ccF$ and $\alpha\in[0,1]$ 
we consider the class $\ccG_{\ccF}(\alpha)$ of $\ccF$-free graphs $G$ with minimum degree at least $\alpha |V(G)|$, i.e., 
\[
	\ccG_{\ccF}(\alpha)=\big\{G\colon\delta(G)\geq \alpha|V(G)| \tand F\nsubseteq G\ \text{for all $F\in\ccF$} \big\}\,,
\]
and for $\ccF=\{F\}$ we simply write $\ccG_F(\alpha)$. Clearly, $\ccG_\ccF(0)$ contains all $\ccF$-free graphs and as~$\alpha$ 
increases the membership in $\ccG_\ccF(\alpha)$ becomes more restrictive. When $\alpha$ is bigger than 
the \emph{Tur\'an density $\pi(\ccF)$}, then $\ccG_\ccF(\alpha)$  contains only finitely many different isomorphism 
types. We are interested in structural properties of members of $G\in\ccG_\ccF(\alpha)$ as $\alpha$ 
moves from~$\pi(\ccF)$ to~0, where structural properties are captured by the existence of 
(graph) homomorphims $G\ahom H$ for some `small' graph $H$.

We begin the discussion with the case of requiring bounded chromatic number, i.e., when~$H$ is allowed to be a 
clique of bounded size (independent of $G$). In that direction,
for $\cF = \{K_{3}\}$,
Erd\H os, Simonovits,  and  Hajnal~\cite{ES73}*{page~325} showed that for every $\eps>0$
there exists a sequence of graphs~$(G_n)_{n\in\NN}$ with members 
from $\ccG_{K_3}(\tfrac{1}{3}-\eps)$ with unbounded chromatic number, i.e.,~$\chi(G_n)\to\infty$ as~$n\to\infty$. In the other direction, Erd\H os and Simonovits conjectured that such a sequence does not exist with members 
from $\ccG_{K_3}(\tfrac{1}{3}+\eps)$. Moving away from the triangle to arbitrary graphs~$F$ (or more generally to families of graphs $\ccF$) this leads to the concept of the \emph{chromatic threshold}
defined by
\[
	\dchi(\ccF) 
	= 
	\inf\big\{\alpha\in[0,1] \colon \text{there is $K\!=\!K(\ccF,\alpha)$ 
		such that $\chi(G)\!\leq\!K$ for every $G\in\ccG_\ccF(\alpha)$}\big\}
\]
and we simply write $\dchi(F)$ for $\dchi(\{F\})$.
The work of Erd\H os, Simonovits,  and  Hajnal then asserts $\dchi(K_3)\geq 1/3$ and 
Erd\H os and Simonovits asked for a matching upper bound. Such an upper bound 
was provided by Thomassen~\cite{Th02} and, therefore, we have 
\begin{equation}\label{eq:dchiK3}
	\dchi(K_3)=\frac{1}{3}\,.
\end{equation}
Addressing another conjecture of Erd\H os and Simonovits from~\cite{ES73} concerning the chromatic threshold of~$C_5$, 
it was also shown by Thomassen~\cite{Th07} that for all odd cycles of length at least~$5$
the chromatic threshold is zero, i.e., $\dchi(C_{2k-1})=0$ for all $k\geq 3$. 
For larger cliques~\eqref{eq:dchiK3} generalises to $\dchi(K_k)=\frac{2k-5}{2k-3}$ for 
all $k\geq 3$ (see~\cites{GL11,Ni}). Extending earlier work of 
\L uczak and Thomass\'e~\cite{LT} and of Lyle~\cite{L11}, eventually 
Allen, B\"ottcher, Griffiths, Kohayakawa, and Morris~\cite{ABGKM13} resolved the general problem and 
determined the chromatic threshold $\dchi(\ccF)$ for every finite family of graphs $\ccF$. 

In the definition of the chromatic threshold $\dchi(\ccF)$ we are concerned with the existence of a small
homomorphic image~$H$ for every $G\in\ccG_{\ccF}(\alpha)$ with $\alpha>\dchi(\ccF)$. However, since we 
allowed~$H$ to be a clique, the homomorphic image 
is not required to be 
$\ccF$-free itself. Adding this additional restriction leads to the following 
definition,
where $H$ is required to be $\ccF$-free as well.
\begin{dfn}\label{def:dhom}
For a family of graphs $\ccF$ we define its \emph{homomorphism threshold}
\begin{multline*}
	\dhom(\ccF) 
	= 
	\inf\big\{\alpha\in[0,1] \colon \text{there is an $\ccF$-free graph $H=H(\ccF,\alpha)$} \\
		\text{such that $G\ahom H$ for every $G\in\ccG_\ccF(\alpha)$}\}\,.
\end{multline*}
If $\ccF=\{F\}$ consists of a single graph only, then we again simply write $\dhom(F)$
.
\end{dfn}

It follows directly from the definition that
\[
	\pi(\ccF)\geq \dhom(\ccF)\geq \dchi(\ccF)
\] 
and that $\dhom(\ccF)=0$ for all families $\ccF$ containing a bipartite graph.
\L uczak~\cite{Lu06} was the first to study the homomorphism threshold
and strengthened~\eqref{eq:dchiK3} by showing that $\dhom(K_3)=\dchi(K_3)=1/3$.
This was extended to larger cliques by Goddard and Lyle~\cite{GL11} and Nikiforov~\cite{Ni}
(see also~\cite{OSch}) and for every $k\geq 3$ we have
\begin{equation}\label{eq:dhomKk}
	\dhom(K_k)=\dchi(K_k)=\frac{2k-5}{2k-3}\,.
\end{equation}
A first step of generalising \L uczak's result by viewing $K_3$ as the shortest odd cycle, 
was recently undertaken by Letzter and 
Snyder~\cite{LS} by showing
\[
	\dhom(C_5)\leq\frac{1}{5}
	\qqand
	\dhom(\{C_3,C_5\})=\frac{1}{5}\,.
\]
We further generalise this result to (families of) cycles of arbitrary odd length
and present the following result.
\begin{thm}\label{thm:main}
	For every integer 
$k \geq 3$ 	
	we have
	\begin{enumerate}[label=\rmlabel]
		\item\label{it:mthm:1} $\dhom(C_{2k-1})\leq \frac{1}{2k-1}$ and 
		\item\label{it:mthm:2} $\dhom(\ccC_{2k-1}) =  \frac{1}{2k-1}$, where the family $\ccC_{2k-1}=\{C_3,C_5,\dots,C_{2k-1}\}$ 
			consists of all odd cycles of length at most 
$2k-1$.
	\end{enumerate}
\end{thm}
Note that for $k=2$ part~\ref{it:mthm:2} of Theorem \ref{thm:main} would include part~\ref{it:mthm:1} and 
this is 
\L uczak's theorem~\cite{Lu06}. 
For $k=3$ Theorem~\ref{thm:main} 
was obtained by 
Letzter and Snyder~\cite{LS}. 
We remark that our approach substantially differs from 
the work of \L uczak and of Letzter and Snyder. For example, \L uczak's proof relied on 
Szemer\'edi's regularity lemma, which is not required here. Moreover, the proof of Letzter and Snyder
is based on a careful case analysis, which yields explicit graphs $H=H(C_5,\alpha)$ for every $\alpha>1/5$ 
(see Section \ref{sec:tetrahedra} for more details).

The lower bound in part~\ref{it:mthm:2} of Theorem~\ref{thm:main} is given by a sequence of generalised Andr\'asfai graphs, 
which we discuss in Section~\ref{sec:Andrasfai}. 
For the 
proof of the 
upper bound of part~\ref{it:mthm:1} we exclude relatively long odd cycles 
in $C_{2k-1}$-free graphs with high minimum degree and we 
specify and prove such a result 
in Section~\ref{sec:preparations}. The proofs of both upper bounds in Theorem~\ref{thm:main}
then follow in Section~\ref{sec:main}. 

\section{Generalised Andr\'asfai graphs}
\label{sec:Andrasfai}
In this section
we establish the lower bound of part~\ref{it:mthm:2} of Theorem~\ref{thm:main}, which will 
be given by a sequence of so-called \emph{generalised Andr\'asfai graphs}.
For $k=2$ those graphs already appeared in the work of  Erd\H os~\cite{Er57}
and were also considered by Andr\'asfai~\cites{A62,A64}.

\vbox{
\begin{dfn}\label{def:AG}
For every integer~$k\geq 2$ we define the class $\ccA_k$ of \emph{Andr\'asfai graphs}
consisting of all graphs $G=(V,E)$, where $V$ is a finite subset of the
unit circle $\RR/\ZZ$ and two vertices 
are adjacent if and only if their distance in $\RR/\ZZ$ 
is bigger than $\tfrac{k-1}{2k-1}$, i.e., 
the neighbourhood of any vertex $v\in V\subseteq \RR/\ZZ$ is given by the set 
$V\cap\big(v+\big(\tfrac{k-1}{2k-1},\tfrac{k}{2k-1}\big)\big)$, 
where 
\[
	v+\big(\tfrac{k-1}{2k-1},\tfrac{k}{2k-1}\big)
	=
	\big\{v+x\colon x\in \big(\tfrac{k-1}{2k-1},\tfrac{k}{2k-1}\big)\big\} 
	\subseteq 
	\RR/\ZZ\,.
\]

Moreover, for integers $k\geq 2$ and $r\geq 1$ the \emph{Andr\'asfai graph}
$A_{k,r}$ is isomorphic to a graph from $\ccA_k$ having the corners of a regular $((2k-1)(r-1)+2)$-gon
as its vertices.
\end{dfn}}

We remark that one can show that every graph $G\in \ccA_k$ is 
homomorphic
to $A_{k,r}$ for sufficiently 
large~$r$. The following properties of Andr\'asfai graphs are 
well-known 
and we include the proof for completeness.

\begin{prop}\label{prop:lbii} 
For all integers $k\geq 2$ and $r\geq 1$ 
the following properties hold
\begin{enumerate}[label=\alabel]
\item\label{it:agraphs1} $A_{k,r}$ is $r$-regular,
\item\label{it:agraphs2} $A_{k,r}$ is $\ccC_{2k-1}$-free,
\item\label{it:agraphs3} if $r\geq 2$ then any two vertices of $A_{k,r}$ are contained in a cycle of length $2k+1$, and
\item\label{it:agraphs4} if $A_{k,r}\ahom H$ for some graph $H$ with  $|V(H)| < |V(A_{k,r})|$, then 
 $H$ contains an odd cycle of length at most $2k-1$. 
\end{enumerate}
In particular, it follows from~\ref{it:agraphs1}, $|V(A_{k,r})|=(2k-1)(r-1)+2$,~\ref{it:agraphs2}, and~\ref{it:agraphs4} that 
$\dhom(\ccC_{2k-1})\geq \frac{1}{2k-1}$.
As $r$ can be chosen arbitrarily big.
\end{prop}
\begin{proof}
For given integers $k\geq 2$ and $r\geq 1$ set 
\[
	n=|V(A_{k,r})|=(2k-1)(r-1)+2
\]
and let $v_0,\dots, v_{n-1}$ be the vertices of $A_{k,r}$ in cyclic order, i.e., we assume 
$v_i\equiv i/n\in\RR/\ZZ$ for every $i=0,\dots,n-1$. By definition of $A_{k,r}$ the neighbourhood
of $v_0$ is contained in the open interval $\big(\frac{k-1}{2k-1},\frac{k}{2k-1}\big)\subseteq \RR/\ZZ$.
Consequently,
\begin{equation}\label{eq:ANv0}
	N(v_0)=\{v_i\colon i=(k-1)(r-1)+1,\dots, k(r-1)+1\}
\end{equation}
and part~\ref{it:agraphs1} follows by symmetry.

For part~\ref{it:agraphs2} we observe that for any closed walk $u_1\dots u_\l u_1$ of length $\l$ 
in $A_{k,r}$ we have $(u_\l-u_1)+\sum_{i=1}^{\l-1}(u_i-u_{i+1})=0$ and 
owing to the definition of $A_{k,r}$ each term of that sum lies in 
$\big(\frac{k-1}{2k-1},\frac{k}{2k-1}\big)\subseteq \RR/\ZZ$.
However, for every integer $j=2,\dots,k$ we have	
\[
	(j-1)\leq (2j-1)\frac{k-1}{2k-1}<(2j-1)\frac{k}{2k-1}\leq j\,.
\]
Consequently, 
$(u_\l-u_1)+\sum_{i=1}^{\l-1}(u_i-u_{i+1}) \in (j-1, j)$. 
Since $0 \notin (j-1, j)$, no walk in $A_{k,r}$ of length $2j-1$ for $j\leq k$ can be 
closed and part~\ref{it:agraphs2} follows.

For part~\ref{it:agraphs3} we show below that starting in $u_0=v_0$ and always choosing the closest 
clockwise neighbour in $A_{k,r}$, i.e., 
setting
\begin{equation}\label{eq:AdefC}
	u_j\equiv u_{j-1}+\frac{(k-1)(r-1)+1}{n}\equiv j\frac{(k-1)(r-1)+1}{n}\in\RR/\ZZ\,,
\end{equation}
defines a Hamiltonian cycle $C=u_0\dots u_{n-1}u_0$ 
with the property that
\[
	u_1\,,\quad u_{(2k-1)+1}\,,\quad u_{2(2k-1)+1}\,,\quad\dots\,,\quad u_{(r-1)(2k-1)+1}=u_{n-1}
\] 
are the $r$ neighbours of $u_0=v_0$ in $A_{k,r}$. 
In other words, every 
$(2k-1)$-th 
vertex on the 
subpath $u_1\dots u_{n-1}$ of the Hamiltonian cycle~$C$ is a neighbour of $u_0$.
Considering the~$C_{2k+1}$'s created by the 
chords between $u_0$ and its neighbours $u_{(2k-1)+1},\dots, u_{(r-2)(2k-1)+1}$ shows that 
$u_0=v_0$ lies on a cycle of length $2k+1$ with every other vertex of $A_{k,\l}$, which by symmetry 
verifies part~\ref{it:agraphs3}.

It is left to show that the~$C$ defined above, has the desired properties
, i.e. is Hamiltonian with the stated distribution of $N(v_{0})$
. It follows from the 
definition of $C$ that $u_{n-1}u_0$ and $u_iu_{i+1}$ are edges of $A_{k,r}$ 
for every $i=0,\dots,n-2$ and, hence, $C$ is a closed walk of length~$n$.
However, since 
\[
	n=(2k-1)(r-1)+2=2\big((k-1)(r-1)+1\big)+(r-1)
\]
and $(k-1)(r-1)+1$ are relatively prime, it follows that $C$ is indeed a Hamiltonian cycle.
Moreover, we observe
for $s=0,\dots,r-1$
that
\begin{align*}
	u_{s(2k-1)+1}
	&\overset{\eqref{eq:AdefC}}{\equiv} 
	(s(2k-1)+1)\frac{(k-1)(r-1)+1}{n}\\
	&\overset{\hphantom{\eqref{eq:AdefC}}}{\equiv}
	(s(2k-1)+1)\frac{(k-1)(r-1)+1}{(2k-1)(r-1)+2}\\
	&\overset{\hphantom{\eqref{eq:AdefC}}}{\equiv}
	\frac{(k-1)(r-1)+1+s}{(2k-1)(r-1)+2}+s(k-1)\\
	&\overset{\hphantom{\eqref{eq:AdefC}}}{\equiv}
	\frac{(k-1)(r-1)+1+s}{n}
	\equiv v_{(k-1)(r-1)+1+s}\overset{\eqref{eq:ANv0}}{\in} N(v_0)\,,
\end{align*}
which shows the 
stated
distribution of $N(v_0)$ on~$C$.

Finally, assertion~\ref{it:agraphs4} is a direct consequence of part~\ref{it:agraphs3}. Suppose 
$\phi\colon A_{k,r}\to H$ is a graph homomorphism with $|V(H)|<n$. Then there are two vertices 
$x$, $y\in V(A_{k,r})$ such that $\phi(x)=\phi(y)$. In particular $xy\not\in E(A_{k,r})$ and in 
view of~\ref{it:agraphs3} the vertex $\phi(x)=\phi(y)$ must be contained in a closed odd walk of 
length at most $2k-1$ in~$H$ and, consequently,~$H$ contains an odd cycle of length at most $2k-1$.
\end{proof}

\section{Dense graphs without odd cycles}
\label{sec:preparations}
In this section we collect a few observation on local properties of graphs with high minimum degree and without an odd cycle of given length.

The main result of this section is the proof of Proposition \ref{prop:CDfree}, which gives some structural information on such graphs by excluding long odd cycles and pairs of odd cycles connected by a path of length~$4$.

We remark that in the following lemmas and in 
Proposition~\ref{prop:CDfree} the additional $\eps n$ term in the minimum degree condition could be replaced by 
some polynomial in~$k$. However, since we do not strive for the optimal condition in these auxiliary results, 
we chose to state them with the same assumption as in Theorem~\ref{thm:main}. We also remark that by the \emph{length} of a path or more generally the length of a walk, we refer to the number of edges, where each 
edge is counted with its multiplicity.
In particular, we denote by $P_{r}$ the path on $r + 1$ vertices.

\begin{lemma}\label{lem:N}
Let $k\geq 2$, $\eps>0$, and let $G=(V,E)$ be a $C_{2k-1}$-free graph satisfying 
$|V|=n \geq 4k/\eps$ 
and $\delta(G)\geq \big(\frac{1}{2k-1}+\eps\big)n$.
	\begin{enumerate}[label=\rmlabel]
	    \item\label{lem:thinN} For every vertex $v \in V$ we have 
$d(M) := 2e(M)/|M| < 2k$ for all $M \subseteq N(v)$.
		\item\label{lem:disjointN} 
For every two vertices $v, u \in V$, i
		f there is an odd $v$-$u$-path of length at most $2k-3$ in $G$, then $u$ and $v$ have less than $5k^{2}$ common neighbours in $G$.
	\end{enumerate}	
\end{lemma}

In the proof of Lemma \ref{lem:N} we will use the following consequence of the Erd\H os-Gallai theorem on paths~\cite{EG59}, also stated in \ref{lem:disjointN}, as well as Theorem 1 of \ref{lem:disjointN}.

\begin{thm}\label{thm:EG-P}{\emph{(Erd\H os \& Gallai 1959)}}
	\begin{enumerate}[label=\rmlabel]
	    \item\label{thm:EG-Pnormal} Let $G$ be an $n$-vertex graph. If $e(G) \geq \frac{1}{2}kn$, then $G$ contains a path with $k$ vertices.
		\item\label{thm:EG-Pbipartite} Let $G = (A, B, E)$ be a bipartite graph with $|A| \geq |B| \geq k$. If $e(A, B) > (|A| + |B|)k$, then $G$ contains an even path of length $k$.
	\end{enumerate}	
\end{thm}

\begin{proof}
Assertion~\ref{lem:thinN} is a direct consequence of Theorem \ref{thm:EG-P} \ref{thm:EG-Pnormal}. 
Indeed
, it implies that 
$d(M)\geq 2k$ yields a copy of $P_{2k-3}$ in $M\subseteq N(v)$, which together with $v$ would form a cycle~$C_{2k-1}$ in $G$.

For the proof of \ref{lem:disjointN} assume for a contradiction that 
$|N(v) \cap N(u)| \geq 5k^{2}$, 
and there is an odd $v$-$u$-path $P$ of length at most $2k-3$.
Let $A' = (N(v) \cap N(u)) \backslash V(P)$, clearly, $|A'| \geq 4k^{2}$ so let $A \subseteq A'$ be a subset of $A'$ with exaktly $4k^{2}$ vertices and $B = N(A) \backslash (A\cup V(P))$.
Since  every vertex in $A$ has at most $2k-2 < 2k$ neighbours in $P$ we have
\[
e(A, B)
	\geq 
|A|\cdot\delta(G) - 2e(A) - |A|\cdot 2k
	\overset{\text{\ref{lem:thinN}}}{>} 
|A| \bigg(\frac{1}{2k-1}n + \epsilon n - 4k\bigg) 
	\geq
\frac{4k^{2}}{2k-1} n
	> 
	2k \cdot n\,.
\] 
Consequently, $|B| > 2k$ and Theorem \ref{thm:EG-P} \ref{thm:EG-Pbipartite} yields a $P_{2k-2}$ in $G[A,B]$ and, hence,
for every $\l\in[k-2]$ there exists a $P_{2\l}$ in  $G[A,B]$ with end vertices in $A$. Together with the path $P$ this yields a cycle $C_{2k-1}$
in $G$, which is a contradiction to the assumption that $G$ is $C_{2k-1}$-free.
\end{proof}

Lemma~\ref{lem:N} yields the following corollary, which asserts
that the first and the second neighbourhoods of a short odd cycle 
cover the ``right'' proportion of vertices.

\vbox{
\begin{lemma}\label{lem:shortC}
Let $k\geq 2$, $\eps>0$, and let $G=(V,E)$ be a $C_{2k-1}$-free graph satisfying $|V|=n \geq 20k^3/\eps$ 
and $\delta(G)\geq \big(\frac{1}{2k-1}+\eps\big)n$. If $C = c_{1}\dots c_{\l}c_1$ is an odd cycle 
of length $\l < 2k-1$ in $G$, then  for every $i\in[\l]$ there are subsets 
$M_{i} \subseteq N(c_{i})\sm V(C)$, vertices $m_{i} \in M_{i}$, and subsets $L_{i} \subseteq N(m_{i})\sm V(C)$
such that the sets $M_{1}, \dots, M_{\l}, L_{1}, \dots, L_{\l}$ 
are mutually disjoint and each of those sets contains at least $\frac{1}{2k-1}n$ vertices.
\end{lemma}}

\begin{proof}
Let $C = c_{1} \dots c_{\l}c_1$ be an odd cycle of length $l$ in $G = (V,E)$, where $l < 2k - 1$.
Since there is a path of odd length at most $\l - 2 < 2k-3$ between any two vertices of~$C$, 
Lemma~\ref{lem:N}~\ref{lem:disjointN} tells us
, that
 $|N(c_{i}) \cap N(c_{j})| < 5k^{2}$ for all distinct~$i$, $j \in [\l]$.
Consequently, we may discard up to at most $(\l-1) \cdot 5k^{2} + \l<10k^3$ vertices from the neighbourhoods~$N(c_{i})$ 
and obtain mutually disjoint sets $M_i\subseteq N(c_i)\sm V(C)$
of size at least 
\[
	\delta(G) - 10k^3 
	\geq
	\frac{1}{2k-1}n + \epsilon n - 10k^3
	> 
	\frac{1}{2k-1}n\,.
\]
For every $i\in[\l]$ fix an arbitrary vertex $m_{i} \in M_{i}$.
Since there is a path of odd length at most $\l - 2 < 2k-3$ between any two vertices of $C$, 
there is a path of odd length at most $(\l - 2) + 2 = \l \leq 2k-3$ between any two 
vertices $m_{i}$ and $m_{j}$. Again we infer from Lemma~\ref{lem:N}~\ref{lem:disjointN} that $|N(m_{i}) \cap N(m_{j})| < 5k^{2}$ for all distinct~$i$, $j \in [\l]$
and in the same way as before, we obtain mutually disjoint sets $L'_{i} \subseteq N(m_{i})\setminus V(C)$ 
of size at least $\delta(G) - 10k^{3}$.

Furthermore, since there also is a path of even length at most $\l - 1 < 2k-3$ between any two (not necessarily distinct) 
vertices of~$C$, 
there is a path of odd length at most $(\l - 1) + 1 = \l \leq 2k-3$ between any pair of vertices 
$c_{i}$ and $m_{j}$. 
Again Lemma~\ref{lem:N}~\ref{lem:disjointN} implies that $|N(c_{i}) \cap N(m_{j})| < 5k^{2}$ for all $i$, $j \in [\l]$
and discarding at most $\l \cdot 5k^{2}<10k^3$ vertices from 
each $L'_{i}$ yields sets $L_{i} \subseteq N(m_{i})$ such that $M_1,\dots,M_\l,L_1,\dots,L_\l$ are mutually disjoint 
and disjoint from $V(C)$.
Moreover,  the assumption $n \geq 20k^{3}/\epsilon$ implies
\[
	|L_i| 
	\geq 
	|L'_i|-10k^3
	\geq
	\delta(G) - 20k^3  
	\geq
	\frac{1}{2k-1}n + \epsilon n - 20k^3 
	\geq 
	\frac{1}{2k-1}n\,,
\]
which concludes the proof of the lemma.
\end{proof}

In the proof of part~\ref{it:mthm:1} of Theorem~\ref{thm:main} it will be useful to exclude 
the graphs described in Definition \ref{def:DL}
as subgraphs 
of a $C_{2k-1}$-free graph of sufficiently high minimum degree.
\begin{dfn}\label{def:DL}
We denote by  $D_{\l}$ the graph on $2\l+3$ vertices that consist of two disjoint cycles of length $\l$ 
and a path of length~$4$ joining these two cycles, which is internally disjoint to both cycles.
\end{dfn}

The following proposition excludes the appearance of some short odd cycles and $D_\l$'s in 
the graphs $G$ considered in Theorem~\ref{thm:main}.

\begin{prop}\label{prop:CDfree}
	Let $k\geq 2$, $\eps>0$, and $G=(V,E)$ be a $C_{2k-1}$-free graph satisfying
	 $|V|=n \geq 20k^3/\eps$ 
	and $\delta(G)\geq \big(\frac{1}{2k-1}+\eps\big)n$. Then  
	\begin{enumerate}[label=\rmlabel]
		\item\label{prop:Cfree} $G$ is $C_\l$-free for every odd $\l$ with $k \leq \l \leq 2k-1$
.
		\item\label{prop:Dfree} $G$ is $D_\l$-free for every odd $\l$ with $\max\{3,2k-7\} \leq \l \leq 2k-1$.
	\end{enumerate}	
\end{prop}

\begin{proof}
Assertion~\ref{prop:Cfree} is a direct consequence of Lemma~\ref{lem:shortC}, as the mutually disjoint sets~$M_1,\dots,M_\l,L_1,\dots,L_\l$
would not fit into~$V(G)$.

For the proof of assertion~\ref{prop:Dfree} we assume for a contradiction that $G=(V,E)$ contains 
a subgraph~$D_{\l}$ for some odd $\l$ with $\max\{3,2k-7\} \leq \l \leq 2k-1$. Since 
the 
graph
$D_{\l}$ contain a cycle of length $\l$, we immediately infer from part~\ref{prop:Cfree}, 
that we may assume $\l < k$. Consequently, $k > \l \geq  2k-7$ implies $k\leq 6$
and owing to $k>\l\geq \max\{3,2k-7\}$ we see that the only remaining cases 
we have to consider are $(k,\l)\in\{(4,3),(5,3),(6,5)\}$. We discuss each of the cases below.
\smallskip

\noindent
\textbf{Case $k=6$ and $\l=5$.}
Let $C=c_1\dots c_5c_1$ and $C'=c'_1\dots c'_5c'_1$ be the two cycles of length~$5$ appearing in 
$D_{5} \subseteq G$ and suppose the path $P$ of length~$4$ connects $c^{}_{1}$ and $c'_{1}$. We observe that 
$c'_5$ is connected to every vertex of $C$ by an odd path of length at most~$9$, as seen in Figure \ref{fig:D5k6}. 
In fact, $Q=c'_5P$ connects 
$c'_5$ and $c^{}_1$ by a path of length~$5$ and every other vertex of $C$ can be reached by an even path of length
at most~$4$ from~$c_1$. 

Furthermore, $c'_5$ is connected to every vertex in $N(C)$ by an odd path of length at most~$9$. For the vertices 
in $N(C)\sm N(c_1)$ we again follow the 
path
~$Q$ and since $c_2,c_3,c_4$, and $c_5$ can be reached by an odd path of length at most~$3$
from~$c_1$, as seen in Figure \ref{fig:D5k6}, every vertex in $N(C)\sm N(c_1)$ can be reached by an odd path of length at most $5+3+1=9$. For the vertices in 
$N(c_1)$ we utilise the path of length~$4$ from $c'_5$ to $c'_1$ in $C'$. Continuing then along $P$ to $c_1$ shows that there are 
paths of length~$9$ connecting $c'_5$ with every vertex in~$N(c_1)$.

As 
$9=2k-3$, we infer from Lemma~\ref{lem:N}~\ref{lem:disjointN} that $c'_5$ has at most 
$10\cdot (5k^2  + |Q|) < 10k^3$ neighbours in
the sets 
$M_1,\dots,M_5,L_1,\dots,L_5$
 given by Lemma~\ref{lem:shortC} applied to~$C$. However, since
\[
	\big|M_1\dcup\dots\dcup M_5\dcup L_1\dcup\dots\dcup L_5\big| \geq \frac{10}{11}n
\]
this implies $\deg(c'_5)\leq \frac{n}{11}+10k^3<\frac{n}{11}+\eps n$ by the 
assumption that 
$n > 20k^3/\eps$, 
which contradicts the minimum degree assumption on $G$ in this case.
\smallskip

\begin{figure}[ht]
\centering
\begin{tikzpicture}
	\draw [line width=12pt, red!80!white, opacity=0.2, rounded corners=5pt, line cap=round] 
		($(-5,0)+(144:2)$) -- ($(-5,0)+(72:2)$) -- ($(-5,0)+(0:2)$) -- ($(5,0)+(180:2)$) -- ($(5,0)+(108:2)$);
	\draw [line width=12pt, blue!80!white, opacity=0.2, rounded corners=5pt, line cap=round] 
		($(-5,0)+(144:2)$) -- ($(-5,0)+(216:2)$) -- ($(-5,0)+(288:2)$) -- ($(-5,0)+(0:2)$) -- ($(5,0)+(180:2)$) -- ($(5,0)+(108:2)$);
	\foreach \x in {1,2,...,5}{
 		\fill (-5,0)+(72-\x*72:2) circle (2.8pt);		
		\draw[very thick] ($(-5,0)+(72-\x*72:2)$) -- ($(-5,0)+(-\x*72:2)$);
		\fill (5,0)+(108+\x*72:2) circle (2.8pt);
		\draw[very thick] ($(5,0)+(108+\x*72:2)$) -- ($(5,0)+(180+\x*72:2)$);
	}	
	\foreach \x in {2,...,5}{
		\node at ($(-5,0)+(72-\x*72:2.45)$) {$c^{}_\x$};
		\node at ($(5,0)+(108+\x*72:2.47)$) {$c'_\x$};
	}
	\node at ($(-5,0)+(10:2.3)$) {$c^{}_1$};
	\node at ($(5,0)+(167:2.3)$) {$c'_1$};
	
	\fill (-1.5,0) circle (2.8pt);
	\fill (0,0) circle (2.8pt);
	\fill (1.5,0) circle (2.8pt);
	
	\draw[very thick]  ($(-5,0)+(0:2)$) -- ($(5,0)+(180:2)$);
\end{tikzpicture}
\caption{An odd path of length $7$ from $c'_{5}$ to $c^{}_{4}$ in red 
	and an even path of length $8$ from $c'_{5}$ to $c^{}_{4}$ in blue as used 
	in the proof of case $k=6$ and $\l=5$.\label{fig:D5k6}}
\end{figure}
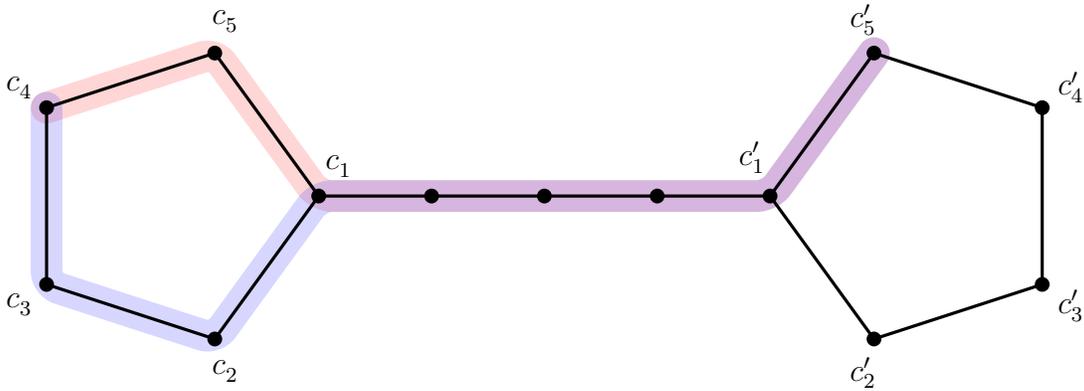

\noindent
\textbf{Case $k=5$ and $\l=3$.}
Let $C=c^{}_1c^{}_2c^{}_3c^{}_1$ and $C'=c'_1c'_2c'_3c'_1$ be the two triangles of $D_3\subseteq G$ and 
suppose the path of length~$4$ connects $c^{}_{1}$ and $c'_{1}$. Moreover, Lemma~\ref{lem:shortC} applied with~$C$ yields 
vertices $m_1,m_2,m_3$ and vertex sets $M_1,M_2,M_3$ and $L_1,L_2,L_3$. It is easy to check that $c'_2$ and $c'_3$
can reach each $c^{}_i$ and $m^{}_i$ for every $i\in[3]$ by an odd path of length at most $7=2k-3$, as seen in Figure \ref{fig:D3k5} on the left
. In view of Lemma~\ref{lem:N}~\ref{lem:disjointN}, and since $|N(c_{2}')|, |N(c_{3}')| \geq \delta(G) \geq n/9$ it follows that

\[
	\big|M_1\cup M_2\cup M_3\cup L_1\cup L_2\cup L_3\cup N(c'_2)\cup N(c'_3)\big|\geq \frac{8}{9}n\,.
\]
Consequently, we infer from $|N(c'_1)|\geq \delta(G)\geq n/9+\eps n>n/9+40k^2$ that the vertex $c'_1$ must have 
at least $5k^2$ common neighbours with one of the eight vertices $c^{}_{1}$, $c^{}_{2}$, $c^{}_{3}$, $m^{}_{1}$, $m^{}_{2}$, $m^{}_{3}$, $c'_{2}$, $c'_{3}$.
Since $c'_1$ can be connected by 
an odd path
of length at most $7$ to all of these eight vertices but~$c_1$, we infer that
$c_1$ and $c'_1$ have $5k^2$ common neighbours and we can fix such a neighbour disjoint from $m_1$, $m_2$, $m_3$, $C$ and $C'$.
In other words, we found a graph~$D'_3$ consisting of~$C$,~$C'$, and a path of length~$2$ between~$c^{}_1$ and $c'_1$.
Consequently, $c'_2$ and $c'_3$ are connected to each $c^{}_i$ and each $m^{}_i$ for every~$i\in[3]$ 
by an odd path of 
length at most~$5$. Hence, we can fix a neighbour $m'_2$ of $c'_2$, which can be connected to each $c^{}_i$ and 
each~$m^{}_i$
for $i\in[3]$ and to~$c'_2$ and $c'_3$ by an odd path of length at most~$7$, as seen in Figure~\ref{fig:D3k5} on the right. In other words, any two of the 9 vertices from 
$c^{}_1,c^{}_2,c^{}_3,m^{}_1,m^{}_2,m^{}_3,c'_2,c'_3$ and~$m'_2$ are connected by an odd path of length at most~$7$
and 
thus 
have 
fewer than 
$5k^2$ common neighbours by Lemma~\ref{lem:N}~\ref{lem:disjointN}.
However, since $\eps n>40k^2$ the minimum degree assumption implies that at least one of the former 8 vertices must have at least $5k^2$ common neighbours with $m'_{2}$.

\begin{figure}[ht]\centering
\begin{tikzpicture}
\node[fill=black,                         circle, inner sep=0pt, minimum size=5pt, name=c1]{};
\node[fill=black, above left  = 8mm of  c1, circle, inner sep=0pt, minimum size=5pt, name=c2]{};
\node[fill=black, below left  = 8mm of  c1, circle, inner sep=0pt, minimum size=5pt, name=c3]{};
\node[fill=black,       right = 5mm of  c1, circle, inner sep=0pt, minimum size=5pt, name=p1]{};
\node[fill=black,       right = 5mm of  p1, circle, inner sep=0pt, minimum size=5pt, name=p2]{};
\node[fill=black,       right = 5mm of  p2, circle, inner sep=0pt, minimum size=5pt, name=p3]{};
\node[fill=black,       right = 5mm of  p3, circle, inner sep=0pt, minimum size=5pt, name=cc1]{};
\node[fill=black, above right = 8mm of cc1, circle, inner sep=0pt, minimum size=5pt, name=cc2]{};
\node[fill=black, below right = 8mm of cc1, circle, inner sep=0pt, minimum size=5pt, name=cc3]{};

\node[above = 1mm of cc1] {$c'_{1}$};

\draw[black] (c1) -- (c2) -- (c3) -- (c1);
\draw[black] (c1) -- (p1) -- (p2) -- (p3) -- (cc1);
\draw[black] (cc1) -- (cc2) -- (cc3) -- (cc1);

\node[draw=black, below = 2mm of c1, ellipse, inner sep=0pt, minimum width=16pt, minimum height=8pt, name=nc1]{};
\node[draw=black, below = 2mm of nc1, ellipse, inner sep=0pt, minimum width=16pt, minimum height=8pt, name=nnc1]{};
\node[draw=black, left = 2mm of c2, ellipse, inner sep=0pt, minimum width=8pt, minimum height=16pt, name=nc2]{};
\node[draw=black, left = 2mm of nc2, ellipse, inner sep=0pt, minimum width=8pt, minimum height=16pt, name=nnc2]{};
\node[draw=black, left = 2mm of c3, ellipse, inner sep=0pt, minimum width=8pt, minimum height=16pt, name=nc3]{};
\node[draw=black, left = 2mm of nc3, ellipse, inner sep=0pt, minimum width=8pt, minimum height=16pt, name=nnc3]{};

\node[draw=black, right = 2mm of cc2, ellipse, inner sep=0pt, minimum width=8pt, minimum height=16pt, name=ncc2]{};
\node[draw=black, right = 2mm of cc3, ellipse, inner sep=0pt, minimum width=8pt, minimum height=16pt, name=ncc3]{};

\node[fill=black, below= 2.5mm of c1, circle, inner sep=0pt, minimum size=5pt, name=m1]{};
\node[fill=black, left = 2.5mm of c2, circle, inner sep=0pt, minimum size=5pt, name=m2]{};
\node[fill=black, left = 2.5mm of c3, circle, inner sep=0pt, minimum size=5pt, name=m3]{};

\draw[black] ($(nc2) + (0mm,2.8mm)$) -- (c2) -- ($(nc2) - (0mm,2.8mm)$);
\draw[black] ($(nnc2) + (0mm,2.8mm)$) -- (m2) -- ($(nnc2) - (0mm,2.8mm)$);
\draw[black] ($(nc3) + (0mm,2.8mm)$) -- (c3) -- ($(nc3) - (0mm,2.8mm)$);
\draw[black] ($(nnc3) + (0mm,2.8mm)$) -- (m3) -- ($(nnc3) - (0mm,2.8mm)$);
\draw[black] ($(nc1) + (2.8mm,0mm)$) -- (c1) -- ($(nc1) - (2.8mm,0mm)$);
\draw[black] ($(nnc1) + (2.8mm,0mm)$) -- (m1) -- ($(nnc1) - (2.8mm,0mm)$);
\draw[black] ($(ncc2) + (0mm,2.8mm)$) -- (cc2) -- ($(ncc2) - (0mm,2.8mm)$);
\draw[black] ($(ncc3) + (0mm,2.8mm)$) -- (cc3) -- ($(ncc3) - (0mm,2.8mm)$);
\end{tikzpicture}
\qquad
\begin{tikzpicture}
\node[fill=black, circle, inner sep=0pt, minimum size=5pt, name=c1]{};
\node[fill=black, above left = 8mm of c1, circle, inner sep=0pt, minimum size=5pt, name=c2]{};
\node[fill=black, below left = 8mm of c1, circle, inner sep=0pt, minimum size=5pt, name=c3]{};
\node[fill=white, right = 5mm of c1, circle, inner sep=0pt, minimum size=5pt, name=p1]{};
\node[fill=black, right = 5mm of p1, circle, inner sep=0pt, minimum size=5pt, name=p2]{};
\node[fill=white, right = 5mm of p2, circle, inner sep=0pt, minimum size=5pt, name=p3]{};
\node[fill=black, right = 5mm of p3, circle, inner sep=0pt, minimum size=5pt, name=cc1]{};
\node[fill=black, above right = 8mm of cc1, circle, inner sep=0pt, minimum size=5pt, name=cc2]{};
\node[fill=black, below right = 8mm of cc1, circle, inner sep=0pt, minimum size=5pt, name=cc3]{};

\draw[black] (c1) -- (c2) -- (c3) -- (c1);
\draw[black] (c1) -- (p2) -- (cc1);
\draw[black] (cc1) -- (cc2) -- (cc3) -- (cc1);

\node[draw=black, below = 2mm of c1, ellipse, inner sep=0pt, minimum width=16pt, minimum height=8pt, name=nc1]{};
\node[draw=black, below = 2mm of nc1, ellipse, inner sep=0pt, minimum width=16pt, minimum height=8pt, name=nnc1]{};
\node[draw=black, left = 2mm of c2, ellipse, inner sep=0pt, minimum width=8pt, minimum height=16pt, name=nc2]{};
\node[draw=black, left = 2mm of nc2, ellipse, inner sep=0pt, minimum width=8pt, minimum height=16pt, name=nnc2]{};
\node[draw=black, left = 2mm of c3, ellipse, inner sep=0pt, minimum width=8pt, minimum height=16pt, name=nc3]{};
\node[draw=black, left = 2mm of nc3, ellipse, inner sep=0pt, minimum width=8pt, minimum height=16pt, name=nnc3]{};

\node[draw=black, right = 2mm of cc2, ellipse, inner sep=0pt, minimum width=8pt, minimum height=16pt, name=ncc2]{};
\node[draw=black, right = 2mm of cc3, ellipse, inner sep=0pt, minimum width=8pt, minimum height=16pt, name=ncc3]{};

\node[fill=black, below= 2.5mm of c1, circle, inner sep=0pt, minimum size=5pt, name=m1]{};
\node[fill=black, left = 2.5mm of c2, circle, inner sep=0pt, minimum size=5pt, name=m2]{};
\node[fill=black, left = 2.5mm of c3, circle, inner sep=0pt, minimum size=5pt, name=m3]{};

\draw[black] ($(nc2) + (0mm,2.8mm)$) -- (c2) -- ($(nc2) - (0mm,2.8mm)$);
\draw[black] ($(nnc2) + (0mm,2.8mm)$) -- (m2) -- ($(nnc2) - (0mm,2.8mm)$);
\draw[black] ($(nc3) + (0mm,2.8mm)$) -- (c3) -- ($(nc3) - (0mm,2.8mm)$);
\draw[black] ($(nnc3) + (0mm,2.8mm)$) -- (m3) -- ($(nnc3) - (0mm,2.8mm)$);
\draw[black] ($(nc1) + (2.8mm,0mm)$) -- (c1) -- ($(nc1) - (2.8mm,0mm)$);
\draw[black] ($(nnc1) + (2.8mm,0mm)$) -- (m1) -- ($(nnc1) - (2.8mm,0mm)$);
\draw[black] ($(ncc2) + (0mm,2.8mm)$) -- (cc2) -- ($(ncc2) - (0mm,2.8mm)$);
\draw[black] ($(ncc3) + (0mm,2.8mm)$) -- (cc3) -- ($(ncc3) - (0mm,2.8mm)$);

\node[fill=black, right= 2.5mm of cc2, circle, inner sep=0pt, minimum size=5pt, name=mm2]{};
\node[below = 1mm of mm2] {$m'_{2}$};
\end{tikzpicture}

\caption{On the left the graph $D_{3}$ of 
the
case $k=5$ and $\l=3$ where the vertex $c'_{1}$ does not have enough neighbours, and on the right the graph $D'_{3}$ where the vertex $m'_{2}$ does not have enough neighbours.\label{fig:D3k5}}
\end{figure}
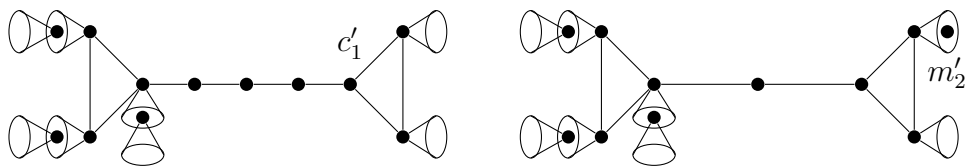

\noindent
\textbf{Case $k=4$ and $\l=3$.}
Again we consider the two triangles $C^{} = c^{}_1c^{}_2c^{}_3c^{}_1$ and $C' = c'_1c'_2c'_3c'_1$ of $D_3\subseteq G$ and assume $c^{}_1$ and $c'_1$ are connected by a path $c^{}_1p^{}_1p^{}_2p^{}_3c'_1$ of length~$4$.
We consider the vertices $m^{}_1$, $m^{}_2$, $m^{}_3$ and sets $M^{}_{1},M^{}_{2},M^{}_{3}$
, $L^{}_{1},L^{}_{2},L^{}_{3}$ 
and $M'_{1},M'_{2},M'_{3}$ given by Lemma~\ref{lem:shortC} applied with $C$ and with $C'$.

Note that there can only be one edge between a vertex of $C^{}$ and a vertex of $C'$, namely~$c^{}_{1}c'_{1}$, otherwise there is a $C_{7}$ in $D_{3}$.
Therefore, if there are vertices $c^{}_{i}$ and $c'_{j}$ with~$i, j \in [3]$ such that they have at least two common neighbours,~$G$ contains a graph $D'_{3}$ consisting of~$C$,~$C'$ and a path of length~$2$ between $c^{}_{i}$ and~$c'_{j}$.
By symmetry, we may assume $i = j = 1$.
However, in this case we see that $c'_2$ is connected to $c^{}_1,c^{}_2,c^{}_3$ and $m^{}_1,m^{}_2,m^{}_3$ by an odd path of length at most~$5$, as seen in Figure \ref{fig:D3k4} on the right
. Since
\[
	\big| M^{}_{1} \cup M^{}_{2} \cup M^{}_{3} \cup L^{}_{1} \cup L^{}_{2} \cup L^{}_{3} \big|
	\geq 
	\frac{6}{7}n\,,
\]
 the minimum degree assumption yields at least $(\eps n - 4)/6\geq 5k^2$ common neighbours of $c'_{2}$ and one of the vertices of $\{c^{}_{1}, c^{}_{2}, c^{}_{3}, m^{}_{1}, m^{}_{2}, m^{}_{3}\}$, which is a contradiction to Lemma~\ref{lem:N}~\ref{lem:disjointN}.

Assuming that no two vertices of $C^{}$ and $C'$ have more than one common neighbour, we notice that $p^{}_1$ can be connected to all three vertices of $C$ and to all three vertices of $C'$ by an odd path of length at most $5 = 2k-3$, as seen in Figure \ref{fig:D3k4} on the left.
Which implies that
\[
	\big|M^{}_1\cup M^{}_2\cup M^{}_3\cup M'_1\cup M'_2\cup M'_3\big|
	\geq 
	\frac{6}{7}n-9\,.
\]
Consequently, the minimum degree assumption yields at least 
$(\eps n - 9 - 9)/6\geq 5k^2$ 
common neighbours of $p^{}_1$ 
and 
one of the vertices of 
$C^{}_1$ or $C'_1$, which is a contradiction to Lemma~\ref{lem:N}~\ref{lem:disjointN}.
\begin{figure}[ht]\centering
\begin{tikzpicture}
\node[fill=black, circle, inner sep=0pt, minimum size=5pt, name=c1]{};
\node[fill=black, above left = 8mm of c1, circle, inner sep=0pt, minimum size=5pt, name=c2]{};
\node[fill=black, below left = 8mm of c1, circle, inner sep=0pt, minimum size=5pt, name=c3]{};
\node[fill=black, right = 5mm of c1, circle, inner sep=0pt, minimum size=5pt, name=p1]{};
\node[fill=black, right = 5mm of p1, circle, inner sep=0pt, minimum size=5pt, name=p2]{};
\node[fill=black, right = 5mm of p2, circle, inner sep=0pt, minimum size=5pt, name=p3]{};
\node[fill=black, right = 5mm of p3, circle, inner sep=0pt, minimum size=5pt, name=cc1]{};
\node[fill=black, above right = 8mm of cc1, circle, inner sep=0pt, minimum size=5pt, name=cc2]{};
\node[fill=black, below right = 8mm of cc1, circle, inner sep=0pt, minimum size=5pt, name=cc3]{};

\node[above = 1mm of p1] {$p_{1}$};

\draw[black] (c1) -- (c2) -- (c3) -- (c1);
\draw[black] (c1) -- (p1) -- (p2) -- (p3) -- (cc1);
\draw[black] (cc1) -- (cc2) -- (cc3) -- (cc1);

\node[draw=black, below = 2mm of c1, ellipse, inner sep=0pt, minimum width=16pt, minimum height=8pt, name=nc1]{};
\node[draw=black, left = 2mm of c2, ellipse, inner sep=0pt, minimum width=8pt, minimum height=16pt, name=nc2]{};
\node[draw=black, left = 2mm of c3, ellipse, inner sep=0pt, minimum width=8pt, minimum height=16pt, name=nc3]{};

\node[draw=black, below = 2mm of cc1, ellipse, inner sep=0pt, minimum width=16pt, minimum height=8pt, name=ncc1]{};
\node[draw=black, right = 2mm of cc2, ellipse, inner sep=0pt, minimum width=8pt, minimum height=16pt, name=ncc2]{};
\node[draw=black, right = 2mm of cc3, ellipse, inner sep=0pt, minimum width=8pt, minimum height=16pt, name=ncc3]{};

\draw[black] ($(nc2) + (0mm,2.8mm)$) -- (c2) -- ($(nc2) - (0mm,2.8mm)$);
\draw[black] ($(nc3) + (0mm,2.8mm)$) -- (c3) -- ($(nc3) - (0mm,2.8mm)$);
\draw[black] ($(nc1) + (2.8mm,0mm)$) -- (c1) -- ($(nc1) - (2.8mm,0mm)$);
\draw[black] ($(ncc2) + (0mm,2.8mm)$) -- (cc2) -- ($(ncc2) - (0mm,2.8mm)$);
\draw[black] ($(ncc3) + (0mm,2.8mm)$) -- (cc3) -- ($(ncc3) - (0mm,2.8mm)$);
\draw[black] ($(ncc1) + (2.8mm,0mm)$) -- (cc1) -- ($(ncc1) - (2.8mm,0mm)$);
\end{tikzpicture}
\qquad
\begin{tikzpicture}
\node[fill=black,                         circle, inner sep=0pt, minimum size=5pt, name=c1]{};
\node[fill=black, above left = 8mm of c1, circle, inner sep=0pt, minimum size=5pt, name=c2]{};
\node[fill=black, below left = 8mm of c1, circle, inner sep=0pt, minimum size=5pt, name=c3]{};
\node[fill=white, right = 5mm of c1, circle, inner sep=0pt, minimum size=5pt, name=p1]{};
\node[fill=black, right = 5mm of p1, circle, inner sep=0pt, minimum size=5pt, name=p2]{};
\node[fill=white, right = 5mm of p2, circle, inner sep=0pt, minimum size=5pt, name=p3]{};
\node[fill=black,       right = 5mm of  p3, circle, inner sep=0pt, minimum size=5pt, name=cc1]{};
\node[fill=black, above right = 8mm of cc1, circle, inner sep=0pt, minimum size=5pt, name=cc2]{};
\node[fill=black, below right = 8mm of cc1, circle, inner sep=0pt, minimum size=5pt, name=cc3]{};

\node[above = 1mm of cc2] {$c'_{2}$};

\draw[black] (c1) -- (c2) -- (c3) -- (c1);
\draw[black] (c1) -- (p2) -- (cc1);
\draw[black] (cc1) -- (cc2) -- (cc3) -- (cc1);

\node[draw=black, below = 2mm of c1, ellipse, inner sep=0pt, minimum width=16pt, minimum height=8pt, name=nc1]{};
\node[draw=black, below = 2mm of nc1, ellipse, inner sep=0pt, minimum width=16pt, minimum height=8pt, name=nnc1]{};
\node[draw=black, left = 2mm of c2, ellipse, inner sep=0pt, minimum width=8pt, minimum height=16pt, name=nc2]{};
\node[draw=black, left = 2mm of nc2, ellipse, inner sep=0pt, minimum width=8pt, minimum height=16pt, name=nnc2]{};
\node[draw=black, left = 2mm of c3, ellipse, inner sep=0pt, minimum width=8pt, minimum height=16pt, name=nc3]{};
\node[draw=black, left = 2mm of nc3, ellipse, inner sep=0pt, minimum width=8pt, minimum height=16pt, name=nnc3]{};

\node[fill=black, below= 2.5mm of c1, circle, inner sep=0pt, minimum size=5pt, name=m1]{};
\node[fill=black, left = 2.5mm of c2, circle, inner sep=0pt, minimum size=5pt, name=m2]{};
\node[fill=black, left = 2.5mm of c3, circle, inner sep=0pt, minimum size=5pt, name=m3]{};

\draw[black] ($(nc2) + (0mm,2.8mm)$) -- (c2) -- ($(nc2) - (0mm,2.8mm)$);
\draw[black] ($(nnc2) + (0mm,2.8mm)$) -- (m2) -- ($(nnc2) - (0mm,2.8mm)$);
\draw[black] ($(nc3) + (0mm,2.8mm)$) -- (c3) -- ($(nc3) - (0mm,2.8mm)$);
\draw[black] ($(nnc3) + (0mm,2.8mm)$) -- (m3) -- ($(nnc3) - (0mm,2.8mm)$);
\draw[black] ($(nc1) + (2.8mm,0mm)$) -- (c1) -- ($(nc1) - (2.8mm,0mm)$);
\draw[black] ($(nnc1) + (2.8mm,0mm)$) -- (m1) -- ($(nnc1) - (2.8mm,0mm)$);
\end{tikzpicture}

\caption{On the left the graph $D_{3}$ of 
the
case $k=4$ and $\l=3$ where the vertex $p_{1}$ does not have enough neighbours, and on the right the graph $D'_{3}$ where the vertex $c'_{2}$ does not have enough neighbours.\label{fig:D3k4}}
\end{figure}
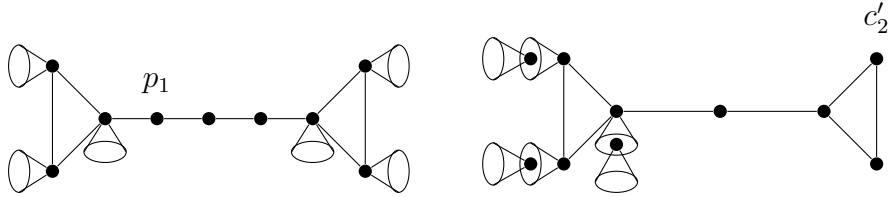
\end{proof}

\section{Upper bounds for Theorem~\ref{thm:main}}
\label{sec:main}

\begin{proof}[Proof of Theorem~\ref{thm:main}]
We first prove assertion \ref{it:mthm:1} of Theorem \ref{thm:main}.
Given a sufficiently large $C_{2k-1}$-free $n$-vertex graph $G = (V, E)$ with $\delta(G) \geq (\frac{1}{2k-1} + \epsilon)n$ for 
$k \geq 3$ and $\epsilon > 0$, 
it suffices to show
that there 
exists a $C_{2k-1}$-free graph $H$ with $|V(H)| \leq K = K(k, \epsilon)$ and $G\ahom H$.
The required graph $H(C_{2k-1},\alpha)$ for Definition~\ref{def:dhom} can then be taken 
to be 
the disjoint 
union of all non-isomorphic $C_{2k-1}$-free graphs on $K$ vertices.

In particular, the constant $K$ must be independent of $n$. Without loss of generality we may assume that $2/\eps$ is an integer.
In order to define $K$, consider the function $f\colon \RR\to \RR$ with $x\mapsto x2^x$ and set 
\begin{equation}\label{eq:defmK}
m=\max \left\lbrace \left \lceil \frac{2\ln(3/\epsilon)}{\epsilon^2} \right \rceil, 8k^2 \right\rbrace
	\qqand
	K = \underbrace{\,f\circ f\circ\dots \circ f}_{\text{$2k$-times}}\big((2/\epsilon + 1)^{\binom{m}{4k}}\big)\,,
\end{equation}
i.e., $K$ is given by a $2(k+1)$-times iterated exponential function in $\textrm{poly}(1/\eps,k)$.

Considering a random $m$-element subsets $X \subseteq V$, it follows from the concentration of the hypergeometric distribution (see e.g.~\cite{JLR}*{inequality (2.6) and Theorem~2.10}) for any fixed vertex $v\in V$
\[
	\PP \big(|N(v) \cap X| \leq (\tfrac{1}{2k-1} + \eps)m - t\big) \leq \exp\big(-\tfrac{t^{2}}{2m}\big)\,,
\]
for every $t > 0$.
Since our choice of $m$ in~\eqref{eq:defmK} yields $m/2k > 4k$ it follows with $t = \eps m$, that there exists a set $X$ of size~$m$, such that all but at most $\epsilon n/3$ vertices of $G$ have at least $4k$ neighbours in $X$.
We fix such a set $X = \{x_{1}, \dots, x_{m}\}$ and set 
\[Y = \big\{v \in V \colon |N(v) \cap X| \geq 4k\big\}\,.\]

For every $y \in Y$ fix a set $X(y)$ of exactly $4k$ neighbours of $y$ in $X$ in an arbitrary way.
We partition $Y$ into $\binom{m}{4k}$ sets, where two vertices $y$, $y'\in Y$ belong to the 
same partition class if $X(y) = X(y')$.
Removing all the classes with 
fewer than 
$8k/\epsilon$ vertices from this partition yields a partition $\cQ$ of a subset of $Y$ of size 
\begin{equation}\label{eq:subQ}
	\left| \bigcup \cQ \right| \geq |Y| - \binom{m}{4k} \frac{8k}{\epsilon} 
	\geq 
	\left(n - \frac{\epsilon}{3} n\right) -  \binom{m}{4k} \frac{8k}{\epsilon} 
	> 
	n - \frac{\epsilon}{2} n\,,
\end{equation}
where the last inequality holds for sufficiently large $n$. For convenience we may index the partition classes of $\cQ$ by a suitable set 
$I=[M]$ with $M\leq \binom{m}{4k}$, i.e., $\cQ= (Q_{i})_{i \in I}$.

Next we define a partition $\cR$ of the whole vertex set $V$, based on the neighbourhoods with respect to the partition classes of $\cQ$.
More precisely we assign to each vertex $v \in V$ a vector $\mu(v) = (\mu_{i}(v))_{i \in I}$, where $\mu_{i}(v)$ equals the proportion of vertices in 
$Q_{i}$ 
that are neighbours of $v$ ``rounded down'' to the next integer multiple of $\epsilon/2$, i.e. 
\begin{equation}\label{eq:mui}
\mu_{i}(v) = \left \lfloor \frac{|N(v) \cap Q_{i}|}{|Q_{i}|} \cdot \frac{2}{\epsilon} \right \rfloor \cdot \frac{\epsilon}{2}\,.
\end{equation}
In particular, since every class from $\cQ$ has at least $8k/\eps$ vertices, we have 
\begin{equation}\label{eq:lbQi}
	\big|N(v)\cap Q_i\big|
	\geq 
	4k
\end{equation}
for every $v\in V$ with $\mu_i(v)>0$.

We now define the partition $\cR$.
The classes of $\cR$ are given by the equivalence classes of the relation $\mu_i(v) = \mu_i(v')$ for every $i\in I$.
Owing to the discretisation of $\mu_{i}(v)$ the partition~$\cR$ has at most 
\[(2/\epsilon + 1)^{|I|} \leq (2/\epsilon + 1)^{\binom{m}{4k}}\] 
parts. Furthermore, we note 
\begin{align}
	\sum_{i \in I} \mu_{i}(v)|Q_{i}| 
	&\overset{\phantom{\eqref{eq:subQ}}}{\geq} d(v) - \left| V\backslash\bigcup \cQ \right| - \sum_{i \in I} \frac{\epsilon}{2} |Q_{i}| \nonumber\\ 
	&\overset{\eqref{eq:subQ}}{>} \left(\frac{1}{2k-1} + \epsilon\right) n - \frac{\epsilon}{2} n 
- \frac{\epsilon}{2} n \nonumber\\
	&\overset{\phantom{\eqref{eq:subQ}}}{\geq} \left( \frac{1}{2k-1} \right) n\label{eq:muQ}
\end{align}
for every $v\in V$. For later reference we make the following observation.

\begin{clm}\label{clm:commonneighbours}
For every $i \in I$ no two distinct vertices $v, v' \in V$ with $\mu_i(v), \mu_i(v') > 0$ are joined by an odd $v$-$v'$-path of length at most $2k-5$ in $G$.
\end{clm}

\begin{proof}
Suppose for a contradiction, that for some $i\in I$ and $v\neq v'$ we have $\mu_i(v),\mu_i(v') > 0$ and there is an odd $v$-$v'$-path~$P$ of length at most $2k-5$ in $G$.
Let~$q_{i}$ be a neighbour of~$v$ in~$Q_i$ and let $q_{i}'$ be a neighbour of $v'$ in $Q_i$, 
such that $q_{i} \neq q_{i}'$ and 
both not contained in~$P$
(see~\eqref{eq:lbQi}).
Consequently, there is a $q_i$-$q'_i$-path $P'\subseteq G$ of odd length $2k-1-2\l$ for some~$\l\in[k-2]$.

Since all vertices of $Q_i$ have $4k$ common neighbours in $X$, 
there is a set $X'$ consisting of~$\l$ of these neighbours from $X\setminus V(P')$. 
Similarly, there is a set $Q'_i\subseteq Q_i$ of $\l-1$ vertices in $Q_i\setminus (V(P')\cup X')$. Clearly, $X'\cup Q'_i\cup\{q_i,q'_i\}$ 
spans a $q_i$-$q'_i$-path $P''$ of length $2\l$, which together with $P'$ yields a copy of $C_{2k-1}$ in $G$.
This, however, contradicts the assumption that $G$ is $C_{2k-1}$-free.
\end{proof}

Starting with 
the 
partition $\cR^{0} = \cR$ we inductively refine this partition $2k$ times and obtain partitions 
$\cR^{0} \succcurlyeq \cR^{1} \succcurlyeq \dots \succcurlyeq \cR^{2k}$. 
Given $\cR^{i}$ we define $\cR^{i+1}$ by subdividing every partition class 
such that vertices remain in the same class if and only if they have neighbours in the same classes of $R^{i}$.
More precisely, two vertices $v$, $v'$ from some partition class of $\cR^i$ stay in the same class in $\cR^{i+1}$ 
if and only if 
for every class $R^i_j$ from $\cR^i$ we have 
\[
	N(v)\cap R^i_j\neq\emptyset  
	\quad\Longleftrightarrow\quad
	N(v')\cap R^i_j\neq\emptyset\,.
\]

Owing to this inductive process and our choice of $K$ in~\eqref{eq:defmK} the partition 
$\cR^{2k}$ consists of at most $K$ classes.
Since $k \geq 3$, claim~\ref{clm:commonneighbours} 
implies that the classes of $\cR^{0}$ are independent sets in $G$ 
and, therefore, also the classes of $\cR^{2k}$ are independent.
Hence, we may define the reduced graph~$H$ of $\cR^{2k}$, where each class 
$\cR^{2k}$ is a vertex of $H$ and two vertices are adjecent, if the
corresponding partition classes induce at least one crossing edge in $G$.  
Obviously, we have 
\begin{equation}\label{eq:goals}
	G\ahom  H
	\qqand
	|V(H)|\leq K
\end{equation}
and it is left to show that $H$ is also $C_{2k-1}$-free (see Claim~\ref{clm:C2k-1}). For the proof of this property 
we first collect a few observations concerning the interplay of odd paths in~$H$ and walks in~$G$ (see Claims~\ref{clm:HtoQ} and~\ref{clm:WtoP}).

Denote by $\cR^{i}(v)$ the unique class of the partition $\cR^{i}$ which contains the vertex $v \in V$.
Similarly, for $j \geq i$ let $\cR^{i}(R)$ be the unique class of the partition $\cR^{i}$ which is a superset of~$R \in \cR^{j}$.

\begin{clm}\label{clm:HtoQ}
If there is a walk $W_H = h_{1} h_{2} \dots h_{s}$ in $H$ for some integer $s \leq 2k$, 
then there are vertices $w_{i} \in \cR^{2k-i+1}(h_{i})\subseteq \cR^0(h_i)$ for every $i\in[s]$ 
such that $W = w_{1} w_{2} \dots w_{s}$ is a walk in $G$. 
Moreover, $w_1$ can be chosen arbitrarily in $h_1=\cR^{2k}(h_1)$.
\end{clm}
\begin{proof}
We shall locate the walk $W$ in an inductive manner and note that for $s=1$ it is trivial.

For $s\geq 2$ 
let a walk $W' =  w_{1} w_{2} \dots w_{s-1}$ satisfying 
 $w_{i} \in \cR^{2k - i+1}(h_{i})$ for every $i\in[s-1]$ be given.
The walk~$W_H$ in~$H$ guarantees an edge between $\cR^{2k}(h_{s-1})$ and $\cR^{2k}(h_{s})$ and, hence, 
there is an edge between  $\cR^{2k-(s-1)+1}(h_{s-1})$ and $\cR^{2k-(s-1)+1}(h_{s})$. Consequently, the 
construction of the refinements shows that $w_{s-1}\in \cR^{2k-(s-1)+1}(h_{s-1})$ must have a neighbour 
$w_s\in \cR^{2k-s+1}(h_{s})$ and the walk $W=W'w_s=w_1\dots w_{s-1}w_s$ has the desired properties. 
\end{proof}

Even if we assume in Claim~\ref{clm:HtoQ} that $W_H$ is a path in $H$ and, in particular,  
$h_i\neq h_j$ for all distinct $i$, $j\in[s]$,
it may happen that~$\cR^{0}(h_{i}) = \cR^{0}(h_{j})$ and, hence, we cannot guarantee
$w_i\neq w_j$. In other words, even if we apply Claim~\ref{clm:HtoQ} to
a path in $H$, the promised walk~$W$ might 
not be a path. However, combined with 
Proposition~\ref{prop:CDfree} we can get the following improvement.

\begin{clm}\label{clm:WtoP}
If there is an odd path $P_H=h_1\dots h_{s+1}$ of length $s\leq 2k-1$ in $H$, 
then there are vertices $v_{1} \in \cR^0(h_1)$ and $v_{s+1} \in \cR^0(h_{s+1})$ such that there is an odd path of length at most~$s$ between them.
\end{clm}

\begin{proof}
Consider a walk $W = w_{1} w_{2} \dots w_{s+1}$ in~$G$ with $w_{i} \in \cR^{0}(h_{i})$  given by Claim~\ref{clm:HtoQ}.
If this walk does not contain an odd $w_1$-\,$w_{s+1}$-path already, then~$W$ must contain an odd cycle.
Below we shall show that this leads to a contradiction and, hence, $W$ contains an 
odd $w_1$-\,$w_{s+1}$-path.

Considering 
an
odd cycle $C = c_{1} \dots c_{\l} c_{1}$ contained in $W\subseteq G$, 
such that
\[
	c_{1} = w_{i_1},\ c_2=w_{i_2},\ \dots,\ c_{\l} = w_{i_\l}, \qqand c_{1} = w_{i_{\l+1}}=w_{i_1}
\] 
for some set of indices satisfying $1 \leq i_{1} < i_2 < \dots < i_{\l} < i_{\l+1}\leq s+1$.
To find such a path, consider the walk $W$, delete an even $w_1$-\,$w_{s+1}$-path.
Now the remaining edges of $W$ are a family of closed walks, take an odd closed walk $W' = w'_{1} w'_{2} \dots w'_{s'+1} = w_{j_{1}} w_{j_{2}} \dots w_{j_{s'+1}}$, where $w'_{1} = w'_{s'+1}$.
If $W'$ is not a cycle, there is a smaller closed walk $W'' \subset W'$, such that the edges of $W'$ without the edges of $W''$ are also a closed walk, both retaining the order of the vertices from $W'$ and therefore from $W$.
One of $\{W', W''\}$ needs to be odd.
Iterating this process eventually gives rise to the odd cycle $C$.

In view of Proposition~\ref{prop:CDfree}~\ref{prop:Cfree} we must have $3\leq \l <k$. 
Consequently, $\l\leq 2k-5$ and since $\l$ is odd, 
it follows from Claim~\ref{clm:commonneighbours} that there is no path of length $\l$ between any two vertices from $\cR^0(c_1) = \cR^0(h_{i_1}) = \cR^0(h_{i_{\l+1}})$. 
Moreover, Claim~\ref{clm:commonneighbours} tells us that the $\l$ classes $\cR^0(c_1) = \cR^0(h_{i_1}) = \cR^0(h_{i_{\l+1}}), \dots, \cR^0(c_{\l})=\cR^0(h_{i_\l})$ from $\cR^0$ are distinct, since otherwise the cycle $C$ would contain an odd path of length at most $2k-7$ between two vertices of  some class in $\cR^0$.

Since $P_H$ is a path in $H$, we have $h_{i_{1}} \neq h_{i_{\l+1}}$ and the cycle $C$ avoids 
at least one of the sets $h_{i_{1}}$ or $h_{i_{\l+1}}$. Without loss of generality we may assume 
$C$ avoids $h_{i_1}$ and we fix an arbitrary vertex $c'_1 \in h_{i_{1}}$.

We are going to locate a second cycle of length $\l$ in $G$ that starts and ends in~$c'_1$. 
By construction this cycle is going to visit the same partition classes of $\cR^{0}$ as $C$.
For that we shall repeat the argument from Claim~\ref{clm:HtoQ} starting with $h_{i_{1}}\dots h_{i_{\l}}h_{i_{\l+1}}$ even though this is not necessarily a subpath of~$P_H$.
However, since $h_{i_{1}}\dots h_{i_{\l}}h_{i_{\l+1}}$ appear in that order in~$P_{H}$, 
we can repeat the reasoning of Claim~\ref{clm:HtoQ} starting with the vertex~$c'_1 \in h_{i_{1}}$. 
Continuing in an inductive manner, for $j\in[\l]$ we have to consider the two cases 
$i_{j+1}=i_j+1$ and $i_{j+1}>i_j+1$.

In the first case, we can indeed proceed as in the proof of Claim~\ref{clm:HtoQ}, since this means that $h_{i_j}h_{i_{j+1}}$ is an edge of~$P_H$. 
The second case
, by construction of $C$, 
only occurs, when $w_{i_j}=w_{i_{j+1}-1}$ and 
\[
	\cR^{2k-i_j+1}(h_{i_j})\subseteq \cR^{2k-(i_{j+1}-1)+1}(h_{i_{j+1}-1})\,.
\] 
Owing to the fact
that $w_{i_{j+1}-1}w_{i_{j+1}}$ is an edge of $W$ and that 
$w_{i_{j+1}-1}\in\cR^{2k-i_j+1}(h_{i_j})$ and~$w_{i_{j+1}-1}\in\cR^{2k-(i_{j+1}-1)+1}(h_{i_{j+1}-1})$, 
we infer from the construction of the refinements 
that $w_{i_j}=w_{i_{j+1}-1}$ also has a neighbour in $\cR^{2k-i_{j+1}+1}(h_{i_{j+1}})$, which concludes the induction step.

Therefore, we obtain another walk $C' = c'_{1} \dots c'_{\l}c'_{\l+1}$ 
where $c'_{j} \in \cR^{0}(h_{i_j})=R^0(c_j)$. Recalling that the $\l$ classes 
$\cR^0(h_{i_1}), \dots, \cR^0(h_{i_\l})$
are pairwise distinct, this implies that $C'$ is either a path or a cycle of odd length~$\l\leq 2k-5$. 
Moreover, since $\cR^0(h_{i_1}) = \cR^0(h_{i_{\l+1}})$ 
we infer from Claim \ref{clm:commonneighbours} that $C'$ cannot be a path and, hence, 
it must be an odd cycle of length $\l\leq 2k-5$.
By construction $c'_1$ avoids~$C$, and hence $C'$ and $C$ are disjoint, as otherwise we would 
have an odd path of length $\l$ connecting $c^{}_1$ and $c'_1$ in $\cR^0(c_1)$, which would
contradict Claim~\ref{clm:commonneighbours} again.

Consequently, $C$ and $C'$ form a copy of~$D_{\l}$ since $c^{}_1$ and $c'_1$ are connected 
by a path of length~$4$ whose three internal vertices avoid $C$ and $C'$ (and the middle vertex is 
from~$X$). 
Owing 
to Proposition \ref{prop:CDfree}~\ref{prop:Dfree} we have $\l \leq 2k-9$, but in $D_\l$ 
there exists an odd path of length $\l + 4 \leq 2k-5$ between $c^{}_i$ and $c'_i$ for every $i=2,\dots,\l$, which again contradicts Claim~\ref{clm:commonneighbours}.
\end{proof}

After these preparations we are now ready to conclude the proof of part~\ref{it:mthm:1} of Theorem~\ref{thm:main}.

\begin{clm}\label{clm:C2k-1}
The graph $H$ is $C_{2k-1}$-free.
\end{clm}
\begin{proof}
Assume for a contradiction that there is a cycle $C_H = h_{1} \dots h_{2k-1} h_{1}$ of length $2k-1$ in $H$.
We recall that the vertices of~$H$ are partition classes of $\cR^{2k}$ and for a simpler notation we set
for any vertex $h_x$ of $C_H$
\[
	\mu_i(h_x)
:=
	\mu_i(v)\,,
\]
where $v$ is an arbitrary vertex from $\cR^0(h_x)$ and the definition of $\cR=\cR^0$ shows that the definition 
of $\mu_i(h_x)$ is indeed independent of the choice of $v\in\cR^0(h_x)$. 

By~\eqref{eq:muQ} we have 
\[
	\sum_{x=1}^{2k-1}\sum_{i\in I}\mu_i(h_x)|Q_i|
	> 
	n
	\geq
	\sum_{i\in I}|Q_i|
\]
and, hence, there is some $i\in I$ such that 
\begin{equation}\label{eq:final}
	\sum_{x=1}^{2k-1}\mu_i(h_x) > 1\,.
\end{equation}
In particular, there are at least two distinct vertices $h_x$ and $h_y$ of $C_H$ such that $\mu_i(h_x)>0$ and 
$\mu_i(h_y)>0$. 
On the other hand, among three vertices of~$C_H$ two are connected by an odd path of length at most $2k-5$ in $C_H$, since the negation is only true for vertices with distance 2 on $C_{H}$. Therefore it follows from Claim~\ref{clm:WtoP} and Claim~\ref{clm:commonneighbours}, that no other vertex~$h_z$ with $z\in[2k-1]\setminus \{x,y\}$ satisfies $\mu_i(h_z)>0$. 
Consequently, we have $\mu_i(h_x)+\mu_i(h_y)>1$, which means that any two vertices 
$v\in \cR^0(h_x)$ and $u\in \cR^0(h_y)$ have a common neighbour in~$Q_i$. In fact, since $2/\eps$ is assumed to be an integer, 
$v$ and~$u$ have at least $2|Q_i|/\eps>4k$ joint neighbours. Moreover, again Claim~\ref{clm:WtoP} and Claim~\ref{clm:commonneighbours} 
imply that $h_x$ and $h_y$ are connected by a path of length~$2k-3$ in~$C_H$ and that there is a path $P$ of length~$2k-3$ in~$G$ connecting some 
$v\in \cR^0(h_x)$ and $u\in \cR^0(h_y)$. Using one of the joint neighbours in~$Q_i$ outside $P$
yields a copy of $C_{2k-1}$ in $G$. This contradicts the $C_{2k-1}$-freeness of $G$ and concludes the proof of 
Claim~\ref{clm:C2k-1}.
\end{proof}

Claim~\ref{clm:C2k-1} together with \eqref{eq:goals} establishes the proof of part~\ref{it:mthm:1} of Theorem~\ref{thm:main}
and 
it remains 
to consider part~\ref{it:mthm:2}, when $G$ is assumed to be $\ccC_{2k-1}$-free.

In view of Proposition~\ref{prop:lbii} it suffices to verify the upper bound of assertion~\ref{it:mthm:2} of Theorem~\ref{thm:main}.
Compared to the proof of part~\ref{it:mthm:1} of Theorem~\ref{thm:main}, we have the additional assumption that~$G$ 
is not only $C_{2k-1}$-free, but also contains no cycle $C_\l$ for any odd~$\l<2k-1$. Consequently, the graph $H$ 
defined in the paragraph before~\eqref{eq:goals} in the proof of part~\ref{it:mthm:1} satisfies~\eqref{eq:goals}
in this case as well 
and owing to Claim~\ref{clm:C2k-1} it is $C_{2k-1}$-free. Hence, we only have to show that the $C_\l$-freeness of~$G$
for every odd $\l\leq 2k-3$ can be carried over to $H$ in this situation, which is rendered by the following claim.

\begin{clm}\label{clm:ccC2k-1}
If $G$ is $\ccC_{2k-1}$-free, then $H$ is also $\ccC_{2k-1}$-free.
\end{clm}
\begin{proof}
Recall, that we  assume $k \geq 3$.
	Suppose for a contradiction that $H$ contains a cycle \mbox{$C_H=h_1\dots h_\l h_1$} for 
	some odd integer $\l$ with $3\leq \l\leq 2k-1$. In fact, it follows from 
	Claim~\ref{clm:C2k-1} that $\l\leq 2k-3$. Moreover, applying Claim~\ref{clm:HtoQ}
	to~$C_H$ yields a walk $W$ of length~$\l$ in $G$ which starts and ends in~$\cR^0(h_1)$. 
	Since~$G$ contains no odd cycle of length at most~$\l$, the walk $W$ contains an 
	odd path of length at most~$\l$ connecting two vertices in $\cR^0(h_1)$. Therefore,
	Claim~\ref{clm:commonneighbours} implies that~$\l=2k-3$ and by symmetry we infer that 
	for every $x\in[2k-3]$ there exists an odd path of length~$2k-3$ between two vertices
	$v_x$, $u_x\in \cR^0(h_x)$.
	
As 
	in the proof of Claim~\ref{clm:C2k-1} we infer from~\eqref{eq:muQ}
that
	\[
		\sum_{x=1}^{2k-3}\sum_{i\in I}\mu_i(h_x)|Q_i|
		>
		\frac{2k-3}{2k-1}n
		>
		\frac{1}{2}\sum_{i\in I}|Q_i|\,,
	\] 
	where we used $k\geq 3$ for the last inequality.
	Consequently, there is some index $i\in I$ such that~$\sum_{x=1}^{2k-3}\mu_i(h_x)>1/2$.
	Since for every distinct $x$, $y\in[2k-3]$ there exists an odd path of length at most $2k-5$ 
	connecting a vertex from $\cR^0(h_x)$ with a vertex from $\cR^0(h_y)$ there is only one vertex of 
	$C_H$ such that $\mu_i(h_x)>0$ and, hence, for that $x\in[2k-3]$ we have 
	$\mu_i(h_x)>1/2$. In particular, every two distinct vertices 
	$v$, $u\in \cR^0(h_x)$ have a common neighbour in~$Q_i$ and, since $2/\eps$ is assumed to be an integer, 
	$v$ and~$u$ have at least $2|Q_i|/\eps>4k$ joint neighbours. Applying this observation to $v_x$ and $u_x$ 
	leads to an odd cycle of length $2k-1$ in $G$, which is a contradiction and concludes the proof of 
	Claim~\ref{clm:ccC2k-1}.
\end{proof}

This concludes the proof of Theorem~\ref{thm:main}.
\end{proof}

\section{Odd Tetrahedra}
\label{sec:tetrahedra}

Letzter and Snyder~\cite{LS} obtained  a stronger version of Theorem~\ref{thm:main}\,\ref{it:mthm:2} for $k=3$,
by showing that the homomorphic images can be chosen from the family of generalised 
Andr\'asfai graphs (see Definition~\ref{def:AG}).
More precisely, it was shown, that $G\ahom A_{3,r}$ for every $G\in\ccG_{\ccC_5}(\alpha)$ as long as 
$\alpha>\frac{r+1}{5r+2}$. 
However, it turns out that such an explicit form of the theorem does not extend to other values of~$k\geq 2$.
For $k=2$ this was observed by H\"aggkvist~\cite{H81}, who showed that there exist appropriate 
(unbalanced) blow-ups of the Gr\"otzsch graph  in $\ccG_{C_3}(10/29)$ which are $4$-chromatic, while 
$\chi(A_{2,r})\leq 3$ for every $r\geq 1$.

In Proposition~\ref{prop:counter} below we provide a counterexample for a stronger version of Theorem~\ref{thm:main}\,\ref{it:mthm:2}
(like the one obtained in~\cite{LS}) for every $k>3$ by exhibiting graphs 
in~$\ccG_{\ccC_{2k-1}}(\frac{1}{2k-1}+\eps)$ for some~$\eps>0$ that are not  
homomorphic to any generalised Andr\'asfai graph from $\ccA_{k}$ (see Definition~\ref{def:AG}).

\begin{dfn}[$(2k+1)$-tetrahedra]
Given $k \geq 2$ we denote by $\ccT_{k}$ the set of graphs $T$ consisting of
\begin{enumerate}[label=\rmlabel]
\item
one cycle $C_{T}$ with three branch vertices $a_{T}$, $b_{T}$, and $c_{T} \in V(C_{T})$,
\item
a center vertex $z_{T}$, and
\item
internally vertex disjoint paths (called spokes) $P_{az}$, $P_{bz}$, $P_{cz}$ connecting the branch vertices with the center vertex.
\end{enumerate}
Furthermore, we require that each cycle in $T$ containing $z_T$ and exactly two of the branch vertices must have length $2k+1$, and the spokes have length at least~$2$.
\end{dfn}

\begin{lemma}\label{prop:tetra}
For all integers $k \geq 2$ and $r \geq 1$ there is no  $(2k+1)$-tetrahedra 
$T \in \ccT_{k}$ that is homomorphic to the Andr\'asfai graph  $A_{k, r}$.
\end{lemma}

\begin{proof}
Let $T \in \ccT_{k}$ be given and let the three spokes consist of $\l_{a}, \l_{b}, \l_{c} \geq 2$ edges, respectively.
Suppose for a contradiction that $T\ahom A_{k, r}$ and let $\phi$ be such a homomorphism.
Since~$T$ contains an odd cycle we have $r\geq 2$ and let $C_{A}=u_0\dots u_{(2k-1)(r-1)+1}u_0$ be the Hamiltonian cycle of $A_{k, r}$ such that $N(u_0)=\{u_{i(2k-1)+1}\colon i=0,\dots,r-1\}$ (c.f.\ proof of Proposition~\ref{prop:lbii}\,\ref{it:agraphs3}). 

\begin{clm}\label{clm:tetra1}
Let $v$, $v'$ be two vertices of a $2k+1$ cycle $C$ in $T$ with distance $d \geq 2$ in $C$.
If~$\phi(v) = u_{0}$, than $\phi(v') \in \{u_{i(2k-1) + d}, u_{i(2k-1) + (2k+1 - d)}\}$ for some integer $0 \leq i \leq r - 2$.
\end{clm}

\begin{proof}
In $C$ there are two paths between $v$ and $v'$ and let $d$ and $d'$ be their lengths.
There cannot be a path of length $d - 2s$ or $d' - 2s$ with $s\geq 1$ between $\phi(v)$ and $\phi(v')$, since this path together with the embedding of the $v$-$v'$-path of other parity from $C$ would form a closed odd walk of length less than $2k+1$, contradicting Proposition \ref{prop:lbii} \ref{it:agraphs2}.
Similarly,~$\phi(v')$ is not in the neighbourhood of $\phi(v)=u_0$ in $A_{k,r}$, since $2\leq d\leq k < d'\leq 2k-1$.

Consequently, $\phi(v') $ will lie on a segment $S$ between $u_{i(2k-1)+1}$ and $u_{(i + 1)(2k-1)+1}$ on the Hamiltonian cycle $C_A$ for some integer $0 \leq i \leq r - 2$. The segment $S$, together with~$u_0=\phi(v)$ 
forms a $C_{2k+1}$, and since there are only two vertices with distance $d$ from~$u_0=\phi(v)$ on this $C_{2k+1}$, an embedding of $v'$ onto any other vertex gives rise to a $v$-$v'$-path of length $d - 2s$ or $d' - 2s$ with $s\geq 1$.
Therefore, $\phi(v') \in \{u_{i(2k-1) + d}, u_{i(2k-1) + (2k+1 - d)}\}$ as claimed.
\end{proof}

\begin{clm}\label{clm:tetra2}
Let $v$, $v'$, $v''$ be distinct vertices of a $2k+1$ cycle $C$ in $T$.
Let $P'$ be the path from $v$ to $v'$ avoiding $v''$ on $C$ and let $P''$ be the path from $v$ to $v''$ avoiding $v'$ on~$C$. Suppose~$d', d'' \geq 2$ are the lengths of $P'$ and $P''$.
If $\phi(v) = u_{0}$, then $\phi(v') = u_{i(2k-1) + d'}$ and $\phi(v'') = u_{j(2k-1) + (2k+1 - d'')}$, 
or $\phi(v') = u_{i(2k-1) + (2k+1 -d')}$ and $\phi(v'') = u_{j(2k-1) + d''}$,
for some integers $0 \leq i, j \leq r - 1$.
\end{clm}

\begin{proof}
By Claim \ref{clm:tetra1} it suffices to show, that $\phi(v') = u_{i(2k-1) + d'}$ implies $\phi(v'') \neq u_{j(2k-1) + d''}$ and $\phi(v') = u_{i(2k-1) + (2k+1 - d')}$ implies $\phi(v'') \neq u_{j(2k-1) + (2k+1 - d'')}$, for all $0 \leq i, j\leq r - 1$.

In the first case, we may assume that $i \leq j$.
Since $u_{j(2k-1) + 2}$ is a neighbour of $u_{i(2k-1)+1}$, we may consider the path $P$ starting with the path in  $C_{A}$ from $u_{i(2k-1) + d'}$ to $u_{i(2k-1)+1}$ together with the edge from $u_{i(2k-1)+1}u_{j(2k-1) + 2}$ and then following $C_{A}$ to $u_{j(2k-1) + d''}$.
The path $P$ consists of $(d' - 1) + 1 + (d'' - 2) = d' + d'' - 2$ edges.
Together with the embedding of the path between $v'$ and $v''$ from  $C$ avoiding $v$, this yields a closed odd walk of length at most $2k-1$ in $A_{k, r}$, contradicting Proposition \ref{prop:lbii} \ref{it:agraphs2}.
A similar argument for the second case concludes the proof of the claim.
\end{proof}
Note that $i(2k-1) + d \neq i(2k-1) + (2k+1 - d)$ for all integers $d, i \geq 0$. Since $z_{T}$ lies in three $C_{2k+1}$, each also containing two of the vertices $a_{T}$, $b_{T}$, $c_{T}$, if $\phi(z_{T})=u_0$, then 
it follows from Claim \ref{clm:tetra2}, that not all three branch vertices can be embedded onto $A_{k, r}$.
Consequently, there is no homomorphism from $T$ to $A_{k,r}$ and Lemma~\ref{prop:tetra} is proved.
\end{proof}

Suitable blow-ups of $(2k+1)$-tetrahedrons show that for every $k\geq 4$ 
there are graphs in~$\ccG_{\ccC_{2k-1}}(\frac{1}{2k-1}+\eps)$ for $\eps>0$ that 
are not homomorphic to $A_{k,r}$ for any $r\geq 1$.

\begin{prop}\label{prop:counter}
	For every integer $k\geq 4$ there is some $\eps>0$ and there are infinitely many graphs 
	in  $\ccG_{\ccC_{2k-1}}(\frac{1}{2k-1}+\eps)$ that are not homomorphic to 
	$A_{k,r}$ for any $r\geq 1$.
\end{prop}
\begin{proof}
Let $k \geq 4$ be fixed and consider the graph $T^{*}$ witch is obtained from 
$K_{4}$, by  replacing two independent edges by a path of length $2(k-3)+1$ 
and the other four edges are replaced by a path of length $3$. In particular, $|V(T^*)|=4k$
and $T^*\in\ccT_{k}$, as all the original triangles of $K_4$ 
are replaced a $C_{2k+1}$. Owing to Lemma~\ref{prop:tetra}, we know that  
$T^{*}$ is not homomorphic to~$A_{k, r}$ and the construction also ensures that 
$T^{*}$ is $\ccC_{2k-1}$-free.

If $k \geq 4$ is even we consider the following blow-ups of $T^{*}$. For every integer 
$f\geq 1$ we consider $T_f^{\rm e}$ obtained from  $T^*$ where the four vertices of degree three and the inner vertices on the two long paths with distance $0\pmod 4$ to one of the two end vertices of the path are replaced by independent sets of size $2f$, while all the other vertices are replaced by independent sets of size $f$.
The graph $T_f^{\rm e}$ is $3f$ regular and has 
\[
	2f\cdot 2(k-2)+f\cdot 2(k+2)
	=
	(6k-4)\cdot f
\]
vertices. Consequently,
\[
	\frac{\delta(T^{\rm e}_f)}{|V(T^{\rm e}_f)|}
	=\frac{3}{6k-4}
	\geq
	\frac{3}{6k-3}+\eps
	= 
	\frac{1}{2k-1}+\eps
\]
for sufficiently small $\eps>0$.
Moreover, since $T^{\rm e}_f$ is a blow-up of $T^*$ it is also $\ccC_{2k-1}$-free and not embeddable into $A_{k, r}$, which shows proves Proposition~\ref{prop:counter} of even integers $k\geq 4$.

For odd integers $k \geq 5$ we also consider blow-ups of $T^*$. For some integer $f\geq 1$ let~$T_f^{\rm o}$ be obtained from $T^*$ by replacing the vertices of degree three and the inner vertices on the two long paths with distance $1\pmod 4$ to one of the end vertices of the path 
by independent sets of size $f$ and all the remaining vertices are kept unchanged.
This blow-up has
\[
	f\cdot2(k-1)+2(k+1)=(f+1)(2k-2)+4
\] 
vertices and minimum degree $f+1$. Consequently, 
\[
	\frac{\delta(T^{\rm o}_f)}{|V(T^{\rm o}_f)|}
	=\frac{f+1}{(2k-2)(f+1)+4}
	\geq
	\frac{1}{2k-1}+\eps 
\]
for sufficiently small $\eps>0$ and sufficiently large $f$.
Again the blow-up $T^{\rm o}_f$ is $\ccC_{2k-1}$-free and not embeddable into $A_{k, r}$, which 
concludes the proof of Proposition~\ref{prop:counter} for odd integers $k\geq 5$.
\end{proof}

\section{Concluding remarks}
\label{sec:remarks}
Theorem~\ref{thm:main} provides only an upper bound for $\dhom(C_{2k-1})$ 
and at this point it is not clear if it is best possible. Proving a matching lower or just 
showing~$\dhom(C_{2k-1})>0$, would require to establish the existence of a sequence 
of graphs~$(G_n)_{n\in\NN}$
with members from $\ccG_{C_{2k-1}}(\alpha)$ for some $\alpha>0$ having no homomorphic $C_{2k-1}$-free
image $H$ of bounded size. However, without imposing~$H$ to be $C_{2k-1}$-free itself, no such sequence exists for $k\geq 3$,
as was shown by Thomassen~\cite{Th07}, 
as the chromatic threshold of odd cycles other than the triangle is $0$, 
which makes the problem somewhat delicate and for the first 
open case we raise the following question.
\begin{question}\label{q:1}
	Is it true that $\dhom(C_5)>0$?
\end{question}
The affirmative answer to Question~\ref{q:1} would, in particular, show that there is a graph~$F$ with 
$\dhom(F)>\dchi(F)$. To our knowledge such a strict inequality is only known for families 
of graphs $\ccF$, like for $\ccF=\ccC_{2k-1}$ for $k\geq 3$.

The lack of lower bounds for families consisting of a single graph, 
may suggest the following natural variation of the homomorphic threshold
\begin{multline*}
	\dhomp(F) 
	= 
	\inf\big\{\alpha\in[0,1] \colon \text{there is an $\ccF$-free graph $H=H(\ccF,\alpha)$} \\
		\text{such that $G\ahom H$ for every $G\in\ccG_F(\alpha)$}\}\,,
\end{multline*}
where $\ccF$ consists of all surjective homomorphic images of~$F$. For odd cycles we have 
$\dhomp(C_{2k-1})=\dhom(\ccC_{2k-1})$ and in view of Theorem~\ref{thm:main} 
it seems possible that $\dhomp(F)$ is easier to determine.

In the proof of Theorem~\ref{thm:main} we showed that every
$G\in\ccG_{C_{2k-1}}(\tfrac{1}{2k-1}+\eps)$ is homomorphic to a $C_{2k-1}$-free 
graph $H$ on at most $K=K(k,\eps)$ vertices, where $K$ is given by a
$2(k+1)$-times iterated exponential function in $\textrm{poly}(1/\eps,k)$. We believe that this
dependency is far from being optimal and maybe already
$K=O(\textrm{poly}(1/\eps,k))$ is sufficient.

In Proposition~\ref{prop:CDfree}~\ref{prop:Cfree} we observed that $C_{2k-1}$-free 
graphs $G$ of high minimum degree are in addition also $C_{2j-1}$-free for some sufficiently large $j< k$
depending on the imposed minimum degree. A more careful analysis of the argument
may yield the correct dependency between~$j$ and the minimum degree of $G$ and, 
moreover, yield a stability version of such a result. However, for a 
shorter presentation we used the same minimum degree assumption as given by Theorem~\ref{thm:main}, which 
sufficed for our purposes. It would also be interesting to see, if the excluded cycles of shorter 
odd length can be also excluded for the homomorphic image~$H$ in the proof of  Theorem~\ref{thm:main}.

Finally, we remark that the blow-ups of tetrahedra considered in Section~\ref{sec:tetrahedra}
are not from $\ccG_{\ccC_{2k-1}}(\frac{1}{2k-2})$. This suggests the question whether for every $k\geq 4$ and 
every $G\in \ccG_{\ccC_{2k-1}}(\frac{1}{2k-2})$ there is some $r\geq 1$ such that $G\ahom A_{k,r}$. 

\subsection*{Acknowledgments}
We we thank both referees for their detailed and helpful remarks.

\end{document}